\documentclass[reqno,12pt]{amsart}
\usepackage{tabularx}
\usepackage{amsmath, amsthm}
\usepackage{amssymb}
\usepackage{enumitem}
\usepackage{xypic}
\usepackage{mathrsfs}
\usepackage{hyperref}
\usepackage{graphicx}
\usepackage{perpage}
\usepackage{extarrows}
\MakePerPage{footnote}
\usepackage{tikz}
\usepackage{setspace}
\date{\today}
\usepackage{longtable}
\hypersetup{colorlinks,linkcolor={blue},citecolor={blue},urlcolor={red}}  
\allowdisplaybreaks
\usepackage{a4wide}   
\usepackage[final]{pdfpages}
    
    \title{Bi-Lipschitz contact invariance of rank}

\author{Nhan Nguyen}

\address{Basque Center for Applied Mathematics (BCAM),
Alameda de Mazarredo 14, 48009 Bilbao, Bizkaia, Spain}
\email{nnguyen@bcamath.org}
\address{ThangLong Institute of Mathematics and Applied Sciences (TMAS), Nghiem Xuan Yem,  Hoang Mai, Hanoi, Vietnam}
\email{nguyenxuanvietnhan@gmail.com}

\newtheorem{thm}{Theorem}[section]
\newtheorem{lem}[thm]{Lemma}
\theoremstyle{theorem}

\newcommand{\bb}{\mathbb}
\newcommand{\al}{\mathcal}

\newcommand{\fs}{\mathscr}

\newcommand{\rank}{{\rm rank}}

\newcommand{\ord}{{\rm ord}}

\begin{document}
\maketitle
\begin{abstract}
In this note we show that the rank of a smooth map is a bi-Lipschitz contact invariant. As a consequence, the first Boardman symbol and its length are bi-Lipschitz contact invariants. 
We also give a counterexample showing that Boardman symbol of a smooth map is not a bi-Lipschitz right invariant. 
\end{abstract}
\section{Introduction}	
In this paper, we focus on the problem of investigating invariants with respect to bi-Lipschitz contact equivalence of mappings. There are not so many results related to this problem have been established, some of them can be found in \cite{Lev2}, \cite{BF}. The bi-Lipschitz contact equivalence implies the bi-Lipschitz equivalence between the zero sets of the mappings, it is, however, not stronger than the topological left-right equivalence. An evidence of this is that in the space of polynomial mappings of fixed  degrees, bi-Lipschitz contact equivalence  does not admit moduli while moduli appears  in case of topological left-right equivalence.  Namely, given a smooth family of polynomial maps,  Ruas and Valette \cite{RV} prove that there is finite decomposition of  the parameter space  such that on each element of the decomposition, the family is bi-Lipschitz contact trivial. Such a decomposition does not always exist for the topological equivalence.  Thom \cite{Thom} gives a family of polynomial maps  $f_t$ in which $f_{t_1}$ and $f_{t_2}$ are not topological left-right equivalent  if $t_1 \neq t_2$.

There are properties which are known in case of bi-Lipschitz contact equivalence but still remain open in the topological case. The multiplicity of complex analytic function germs is an example.  In the same manner, we show in the paper that the rank of a smooth map germ is a bi-Lipschitz contact invariant (Theorem \ref{thm_k}) for both the real and the complex case.  By the rank of a map we mean the rank of its Jacobian matrix. We do not know if rank is a topological left-right (resp. right) invariant in case of  the complex analytic mappings.  For complex analytic functions, it easy to see that there are only two possibilities of rank which is equal either to $1$ or  to $0$. If the rank is $1$, then the germ has no singularity hence the Milnor number is  $0$, otherwise, the Milnor number is bigger than $0$. Since the Milnor number is  a topological invariant, so is the rank.  Rank of smooth functions in the real case is not a topological right invariant.  For instance, consider $f(x) = x$ and $g(x) = x^3$. It is obvious that $f$ and $g$ are topological right equivalent, however, $\rank (f) = 1$ while $\rank (g) = 0$.

Our result is deduced from an interesting property that if two smooth map germs are bi-Lipschitz left-right equivalent then their first homogeneous parts are also bi-Lipschitz left-right equivalent (Theorem \ref{thm_h_a}).  The main tool used in the paper is  the construction of maps associated to a given bi-Lipschitz homeomorphism due to Sampaio \cite{Sampaio}. These maps somehow play a role as tangent maps. 

Together with the order, the rank could be regarded  as one of the most basic invariants for the bi-Lipschitz contact equivalence.  A consequence of this result is that the first Boardman symbol and its length are bi-Lipschitz contact invariants. We would like to remark that Boardman symbol is a very important invariant in the study smooth mappings. It is well known that Boardman symbol is a smooth contact invariant. It would be interesting to know whether Boardman symbol is a bi-Lipschitz invariant.
In Section \ref{section4},  we give a counterexample which shows that in general it is not a bi-Lipschitz right invariant . There is also an example showing that Boardman symbol is not a good enough  to determine the topological type of function germs even for  isolated plane curve singularities, i.e., in a plane, two function germs with isolated singularities of the same Boardman symbol may not have the same topological type.

Throughout the paper, by a smooth map germ $f : \bb K^n, 0 \to \bb K^p, 0$  we mean $f$ is  a $C^\infty$-map in the case $\bb K = \bb R$ and $f$ is  an $analytic$ map in the case $\bb K  = \bb C$.

\section{Preliminaries}
\subsection{Bi-Lipschitz contact equivalence}

Let $\bb K = \bb C$ or $\bb R$.  Let $f, g : \bb K^n, 0 \to \bb K^p, 0$ be smooth map germs. The maps $f$ and $g$ are called bi-Lipschitz left-right equivalent (or bi-Lipschitz $\al A$-equivalent) if there are germs of  bi-Lipschitz homeomorphisms $\varphi: \bb K^n, 0 \to \bb K^n, 0$ and $\psi: \bb K^p, 0 \to \bb K^p, 0$ such that $f\circ \varphi = \psi \circ g$. If $\psi$ coincides with  the identity map, then $f$ and $g$ are called bi-Lipschitz right equivalent (or bi-Lipschitz $\al R$-equivalent).

The maps $f$ and $g$  are \textit{bi-Lipschitz contact equivalent} (or bi-Lipschitz $\fs K$-equivalent) if there are germs of bi-Lipschitz homeomorphisms  $h: \bb K^n, 0 \to  \bb R^n, 0$ and $H : \bb K^n \times \bb K^p, 0 \to  \bb K^n \times \bb K^p, 0$  of the form $H(x, y) = (h(x), \theta(x,y))$  such that  $H(x, f(x)) = (h(x), g(h(x)))$ and  $\theta(x, 0) = 0$.  If $h$ is the identity map then $f$ and $g$ is called bi-Lipschitz $\fs C$-equivalent.

Given a smooth function germ $\xi: \bb K^n, 0 \to \bb K, 0$ .  We may write the Taylor expansion of $\xi$ at $0$ as  $T_0 \xi= \sum_k \xi_k $ where $\xi_k$ is a homogeneous polynomial of degree $k$ . The smallest  integer $k$  such that $\xi_k \neq 0$ is called {\it the order}  (or the multiplicity) of $h$, denoted by $\ord (\xi)$.

For a smooth map germ $f$ given as above, we  may write $f= (f_1, \ldots, f_p)$. The {\it rank of $f$}, denoted $\rank (f)$, is the rank of the Jacobian matrix $(\partial f_i/ \partial x_j)$ at $0$. The order of $f$ is  $\ord (f) = \min \{\ord (f_i)\}_{i = 1, \ldots, p}$ . {\it The first homogeneous part} of $f$ is the polynomial map  $H_f = (f_{1,  \ord (f)}, \ldots, f_{p,  \ord (f)})$ where $f_{i, \ord (f)}$ is the homogeneous polynomial of degree $\ord (f)$ in the Taylor expansion of $f_i$ at $0$.  For example if $f(x, y) = (x^2 + y^3, x^2y)$, then $\rank(f) = 0$,  $\ord (f) = 2$ and $H_f(x, y) = (x^2, 0)$.

We need the following lemmas for the next section. Results in these lemmas are well known and easy to prove.  One might find  an analogous statement of Lemma \ref{lem_1} for function germs in \cite{Lev} and  a proof of Lemma \ref{lem2.2}  in \cite{AF}, Theorem 3.2. For convenience, we provide here with proofs.
\begin{lem} \label{lem_1} If $f$ and $g$ are bi-Lipschitz $\al K$-equivalent, then there is a germ of bi-Lipschitz homeomorphism 
$\varphi : \bb K^n, 0 \to \bb K^n, 0$ and a constant $c > 1$ such that 
$$ \frac{1}{c} \|f\| \leq \|g\circ \varphi \| \leq c \|f\|$$
\end{lem} 
\begin{proof}
By the hypothesis, there is a bi-Lipschitz homeomorphism $\varphi : \bb K^n, 0 \to \bb K^n, 0$ such that $f$ and $\bar{g} = g \circ \varphi$ are bi-Lipschitz $\fs C$-equivalent. It means there is a germ of a bi-Lipschitz homeomorphism $H: \bb K^n \times \bb K^p, 0 \to \bb K^n \times \bb K^p, 0$ such that $H(x, f(x)) = (x, \bar{g}(x))$. Since $f(0) = \bar{g}(0)= 0$,   for any $x$ in a small neighborhood of $0$, 
\begin{align*}
 \|f(x)\| = \|(x, f(x)) - (x, f(0))\| & = \| H^{-1} (x, \bar{g}(x)) -  H^{-1} (x, \bar{g}(0)) \| \\
 & \leq c  \|(x, \bar{g}(x)) - (x, \bar{g}(0))\|   = c\|g(x)\|
\end{align*}
where $c$ is the Lipschitz constant of $H$.
Similarly, we also have $\| \bar{g}(x) \| \leq c \|f(x)\|$.  The lemma is proved.
\end{proof}

\begin{lem}\label{lem2.2} If $f$ and $g$ are bi-Lipschitz $\al K$-equivalent, then $\ord(f) = \ord(g)$.
\end{lem}
\begin{proof}
The proof is similar to the proof of Lemma 4.2 in \cite{nrs}. By Lemma \ref{lem_1}, there is a germ of bi-Lipschitz homeomorphism $\varphi : \bb K^n, 0 \to \bb K^n, 0$  such that $\|f\| \sim \| g\circ \varphi \|$. Write $\varphi = (\varphi_1, \ldots, \varphi_n)$. Note that $\varphi_i$ are Lipschitz function germs. Suppose that $\ord (f) = r$ and $\ord (g) = s$.  By a linear change of coordinates, we may assume that 
$H_f (1,0, \ldots, 0) \neq 0$ and $H_g(1,0, \ldots, 0) \neq 0$ where $H_f$ and $H_g$ are the first homogeneous parts of $f$ and $g$ respectively. Restricting to the $x_1$-axis we have

$$ 1 \sim \frac{|f(x)\|}{\|g(\varphi(x))\|} \sim \frac{\|H_f(x)\|}{ \|H_g(\varphi(x))\|} \sim  \frac{\|x_1\|^r}{\|\varphi_1(x_1, 0, \ldots, 0)\|^s}  \gtrsim \frac{\|x_1\|^r}{\|x_1\|^s} .$$
This implies that $r \geq  s$. Exchanging $f$ with $g$ we obtain $r \leq s$. Therefore, $r = s$.
\end{proof}

\subsection{Pseudo-tangent maps associated to bi-Lipschitz homeomorphisms}
Given a germ of a bi-Lipschitz homeomorphism $\varphi: \bb K^n , 0 \to  \bb K^n, 0$. Suppose that $\psi$ is the inverse map germ of $\varphi$.  For $m \in \bb N$, define 
$$ \varphi_m (x) = m \varphi(\frac{x}{m}), \hspace{1cm} \psi_m (x) = m \psi(\frac{x}{m}).$$
It is obvious that $\varphi_m, \psi_m$ are bi-Lipschitz homeomorphisms of the same Lipschitz constants as $\varphi$ and $\psi$ respectively. By the Arzela-Ascoli theorem, there exists a subsequence $\{m_i\}$ such that $\varphi_{m_i}$ and $\psi_{m_i}$ uniformly converge to bi-Lipschitz maps $\varphi^*$ and $\psi^*$. E. Sampaio showed in \cite{Sampaio}, Theorem 3.2 that $\psi^*$ is the inverse of $\varphi^*$.  We call such  $\varphi^*$  a {\it pseudo-tangent map} associated to the bi-Lipschitz homeomorphism $\varphi$.

\section{Bi-Lipschitz contact invariance of rank}

Let $f, g: \bb K^n, 0 \to \bb K^p, 0$ be smooth map germs.
We have the following results.

\begin{thm} \label{thm_h_a} If $f$ and $g$ are bi-Lipschitz $\al A$-equivalent, so are $H_f$ and $H_g$. 
\end{thm}
\begin{proof} Since $f$ and $g$ are bi-Lipschitz $\al A$-equivalent there are germs of bi-Lipschitz homeomorphisms $\varphi : \bb R^n, 0 \to\bb R^n, 0$ and $\psi : \bb R^p, 0 \to \bb R^p, 0$  such that  $f\circ \varphi = \psi \circ g$. 

By Lemma \ref{lem2.2} we may assume $\ord (f) = \ord (g) = k$. We can write  $f$ and $g$ as
\begin{equation*}\label{eq_lem_cone}
f(x) = H_f(x) + O(\|x\|^{k+1}) \text{ and } g(x) = H_g(x) + O(\|x\|^{k+1})
\end{equation*}  where $k = \ord (f) = \ord (g)$.

By the Arzela-Ascoli theorem, there exists a sequence $\{m_i\}$ of positive integers such that as $i$ tends to $\infty$, the sequence of maps $\{m_i\varphi(\frac{x}{m_i})\}$ converges to a Lipschitz map germ $\varphi^*$ and the sequence of maps $m_i^k \psi (\frac{x}{m_i^k})$ converges to a Lipschitz map $\psi^*$. We also know that $\psi^*$ is the inverse of $\varphi^*$.   Since $H_f$ and $H_g$ are homogeneous polynomial maps of degree $k$, we have

(1) $m_i^k H_f (\varphi(\frac{x}{m})) = H_f(m_i\varphi(\frac{x}{m_i}) \to H_f (\varphi^*(x)) \text{ as } m_i \to \infty.$ 

(2) $m_i^k \psi (H_g (\frac{x}{m_i}) )= m_i^k \psi (\frac{H_g(x)}{m_i^k}) \to  \psi^* (H_g(x)) \text{ as } m_i \to \infty$.

We claim that (i) $m_i^k f(\varphi (\frac{x}{m_i}) )\to H_f (\varphi^*(x))$  and  (ii) $m_i^k \psi(g(\frac{x}{m_i})) \to  \psi^* (H_g(x))$  as $m_i \to \infty$. 
Let us give a proof for the claim.

 (i) 
\begin{align*}
m_i^k f(\varphi (\frac{x}{m_i})) & = m_i^k \left [H_f (\varphi(\frac{x}{m_i})) + O(\|\frac{x}{m_i}\|^{k+1})\right]\\
& =H_f(m_i \varphi(\frac{x}{m_i})) + m_i^k O(\|\frac{x}{m_i}\|^{k+1}) \to H_f (\varphi^*(x)). 
\end{align*} 

(ii)
\begin{align*}
\|m_i^k \psi(g(\frac{x}{m_i})) - m_i^k \psi((H_g(\frac{x}{m_i}))\| & = C m_i^k   \| g(\frac{x}{m_i})- H_g(\frac{x}{m_i}) \| \\
& =C m_i^k  O (\|\frac{ x}{m_i}\|^{k+1} )\to 0,
\end{align*} 
where where  $ C$  is the Lipschitz constant of $\psi$. 
This implies (ii).  

Since 
$$m^k \left [f (\varphi  (\frac{x}{m_i} ))\right ] =m^k \psi (g (\frac{x}{m_i})), $$ thus their limits

$$H_f(\varphi^*(x)) = \psi^* (H_g(x)).$$
The proof completes.

\end{proof}

\begin{thm}\label{thm_k} If $f$ and $g$ are bi-Lipschitz $\al K$-equivalent, then $rank (f) = rank (g)$. 

\end{thm}
\begin{proof} 
Assume that $\rank(f) = r$ and $\rank (g) = s$. By smooth changes of coordinates of the source and the target we can assume that $f(x) = (x_1, \ldots, x_r, \tilde{f}_{r+1}(x), \ldots, \tilde{f}_p(x))$ and $g(x) = (x_1, \ldots, x_s, \tilde{f}_{s+1}(x), \ldots, \tilde{f}_p(x))$ where the orders of  $\tilde{f}_i$ and $\tilde{g}_j$ are bigger than $1$.  

By the hypothesis, there are germs of bi-Lipschitz homeomorphisms   $h: \bb K^n \times \bb K^n$ and $H : \bb K^n \times \bb K^p, 0 \to  \bb K^n \times \bb R^p$ , $H(x, y) = (h(x), \theta(x,y))$  such that
 $H(x, f(x)) = (h(x), g(h(x)))$ and  $\theta(x, 0) = 0$.

Set $F(x) = (x, f(x))$ and $ G(x) = (x, g(x))$. It is obvious that $ H\circ F = G \circ h$. Therefore, $F$ and $G$ are bi-Lipschitz $\al A$-equivalent. By Theorem \ref{thm_h_a},  $H_F = (x, f') $ and $H_G = (x, g')$ are bi-Lipschitz $\al A$-equivalent where $f'(x) = (x_1, \ldots, x_r, 0, \ldots, 0)$ and $g'(x) = (x_1, \ldots, x_s, 0,\ldots,0)$. More precisely, it is shown in the proof of Theorem \ref{thm_h_a}  that there is a pseudo-tangent map $H^* = (h^*, \theta^*)$ associated to the map $H$ where $h^*$ is a pseudo-tangent map associated to $h$ such that $H^*\circ H_F  = H_G\circ h^*$.  Note that $\theta^*(x, y)  = \lim_{m_i \to \infty} m_i \theta(\frac{x}{m_i}, \frac{y} {m_i})$ where $\{m_i\}$ is some sequence of integers tending to infinity. Since $\theta(x, 0) = 0$,   then $\theta^*(x, 0) =0$. This implies that $f'$ and $g'$ are bi-Lipschitz $\al K$-equivalent. Hence, $\{f'= 0\}$ and  $\{g' = 0\}$ are bi-Lipschitz equivalent as sets.  Therefore, $r$ must be equal to $s$. 

\end{proof}

\section{Remarks on Boardman symbol of mappings}\label{section4}
In this section, we would like to give some remarks on the Boardman symbol of smooth map germs.  Let us  first briefly recall the definition of Boardman symbol following the book of Gibson \cite{Gibson}.
Let $\al E_n$ denote the set of smooth function germs  at $0 \in \bb K^n$.  Let $I$ be a finitely generated ideal in $\al E_n$  and $f_1, \ldots, f_p$ be the generators of $I$ and let  $ x_1, \ldots, x_n$ be a system of coordinates in $\al E_n$. For an integer $s \geq 1$, define by $\Delta_s I$ to be the ideal $I + I_s$ where $I_s$ is the ideal generated by $s \times s$-minors of the Jacobian matrix $(\partial f_i / \partial x_j)$. Note that the ideal $\Delta_s I$ does not depend on the choice of generators as well as the choice of coordinates. One then has the inclusion of ideals 
$$ I \subseteq \Delta_n I \subseteq  \Delta_{n-1} I \subseteq \ldots \subseteq\Delta_1 I .$$

Set $\Delta^s I = \Delta_{n - s + 1}$.  It turns out that
\begin{equation*} 
I \subseteq \Delta^1 I \subseteq  \Delta^{2} I \subseteq \ldots \subseteq\Delta^n I \hspace{1cm} (*)
\end{equation*}
The critical Jacobian extension of $I$ the largest proper ideal  $\Delta^{i_1} I$ in  $(*)$.  Continuing the process,  we obtain ascending sequence  $\Delta^{i_1} I, \Delta^{i_2} \Delta^{i_1} I, \ldots$.  The non-increasing sequence $(i_1, i_2, \ldots)$ is called the {\it Boardman symbol} of the ideal $I$.

Let $f : \bb K^n, 0 \to \bb K^p, 0$ be a smooth map germ. The Boardman symbol of $f$, denoted by $\al B(f)$,  is the Boardman symbol of the ideal generated by its components $f_1, \ldots , f_p$ .  
We can write

$$ \al B(f) = (\underbrace{a_1, \ldots, a_1}_{\alpha_1  \text{ times}}, \underbrace{a_2, \ldots, a_2}_{\alpha_2  \text{ times}}, \ldots)$$
where $a_i > a_{i+1} \geq 0$.  We call $a_i$ the {\it $i$-th Boardman symbol} of $f$ and $\alpha_i$ {\it the length} of $a_i$. It is deduced from the definition that 

(i) $a_1 = n - \rank (f)$, and 

(ii) $\alpha_1 = \ord(f)-1$.  

As proved in Theorem \ref{thm_k} that the rank and the order of a smooth map are bi-Lipschitz contact-invariants, so are the first Boardman symbol and its length.

It is known that the Boardman symbol  is a smooth contact invariant (see \cite{Gibson} for example). 

It is natural  to ask the following questions:

 (1) Is Boardman symbol a bi-Lipschitz  right (resp. left-right, contact) invariant?
 
 (2) Is Boardman symbol strong enough to determine the topological type of map germs, i.e., given two map germs of the same Boardman symbol, are they of the same topological type?
 
 Unfortunately, the answers to these questions are negative. Indeed, for the first question let us consider the family of function germs $f_t(x, y)  = x^4 + tx^2y^6 +y^9$. Using Theorem 7.9, \cite{nrs}, it is easy to check that this family is bi-Lipschitz right trivial. Calculation gives  that $\al B(f_0) = (2, 2, 2, 1,1,1,1,1, 0,\ldots)$ and $\al B(f_t) = (2, 2, 2, 1, 1,1,1,0,\ldots)$ for $t\neq 0$.  
 
 For the second question, consider the following complex analytic function germs: $f(x, y) = x^4 + y^5$ and $g(x, y) = x^4 - 2x^2y^3  - 4xy^5 + y^6 + y^7$.  We have  $\al B(f) = \al B(g) = (2, 2, 2, 1, 0,\ldots)$.  The Puiseux pairs of $f$ and $g$ are  $\{(5,4)\}$ and $ \{ (3, 2); (7, 2)\}$ respectively. Since $f$ and $g$ are reduced, their zero sets are not topological equivalent (see \cite{Zariski}, \cite{Parusinski}). 

\subsection*{Acknowledgments} The research was supported by the ERCEA 615655 NMST Consolidator Grant and also by the Basque Government through the BERC 2018--2021 program and by the Spanish Ministry of Science, Innovation and Universities: BCAM Severo Ochoa accreditation SEV-2017--0718.

\end{document}